\documentclass[12pt]{article}

\usepackage[margin=1in]{geometry}
\usepackage{amsmath,amssymb,amsthm,mathtools,bm,color}
\usepackage{enumitem}
\usepackage{microtype}
\usepackage{hyperref}
\newcommand{\N}{\mathbb N}

\newcommand{\R}{\mathbb R}

\newcommand{\E}{\mathbb E}

\newcommand{\VEC}[1]{\mathbf{#1}}

\DeclareMathOperator{\inverseGauss}{IG}
\theoremstyle{plain}
\newtheorem{theorem}{Theorem}[section]

\newtheorem{definition}[theorem]{Definition}

\theoremstyle{remark}
\newtheorem{remark}[theorem]{Remark}

\usepackage{todonotes}

\title{Stable limits for the clone-mutant chain with mother-dependent mutation}
\author{
	Airam Blancas\thanks{Department of Statistics, ITAM, Mexico.
		Email: \texttt{airam.blancas@itam.mx}}
	\and
	Mar\'ia Clara Fittipaldi\thanks{Facultad de Ciencias, UNAM, Mexico.
		Email: \texttt{mcfittipaldi@ciencias.unam.mx}}
}

\date{}

\begin{document}
\maketitle

\begin{abstract}
We consider the general mother-dependent neutral mutation model defined in~\cite{BFH2}, consisting of a multitype branching process in which every individual carries one of $d$ alleles and each mutant child acquires a type different from that of its mother, according to a prescribed distribution. We assume that the total offspring distribution is critical and belongs to the domain of attraction of an $\alpha$-stable law, with $\alpha\in(1,2)$, while mutations are rare. Extending the results of Blancas and Rivero~\cite{BlancasRivero}, we prove the finite-dimensional convergence of the rescaled Markov chain that records, in each allelic generation, the sizes of the type-specific clone subfamilies and the mutant populations they produce. The limiting clone masses form a multitype Ji\v{r}ina process driven by stable subordinators with drift, whose branching mechanism combines the stable reproduction law with the mother-dependent mutation matrix.
\end{abstract}

\section{Introduction}
The general mother-dependent neutral mutations model, defined in~\cite{BFH2}, describes the evolution of a population across discrete generations, where every individual carries one of $d$ alleles and each mutant child acquires a type different from that of its mother, according to a prescribed distribution. The case where the prescribed distribution is uniform corresponds to the mother-dependent neutral mutation model studied in~\cite{BFH, BFH1}. 

In~\cite{BFH1}, we considered a critical population with finite offspring variance and analyze the genealogy of its allelic subfamilies, extending the framework of Bertoin~\cite{Bertoin10} from the infinite-allele setting to finitely many alleles.
In particular, we study the \emph{clone-mutant chain}, that is, the multidimensional Markov chain that records the sizes of the allelic subfamilies of each type together with the numbers of mutants that they beget. In a regime of large initial population and rare mutations, we proved that the rescaled chain converges to a discrete-time continuous-state Markov process driven by the masses of the mutant populations. From the first allelic generation onward, its clone component is a multitype Ji\v{r}ina process whose branching mechanism is the Laplace exponent of a $d$-dimensional inverse Gaussian subordinator carried by the coordinate axes, see \cite[Theorem 2.8]{BFH1}. Lifting this convergence to the genealogical level, we introduced a $d$-type tree of alleles and showed that, in the scaling limit, it converges to a tree-indexed multitype Ji\v{r}ina process, described by the ranked atoms of Poisson random measures with inverse Gaussian intensity.

As shown in \cite[Lemma 3.5]{BFH2}, the clone-mutant process associated with the general mother-dependent neutral mutation model can be obtained from the corresponding infinite-alleles BGW process, studied in \cite{Bertoin10} and \cite{BlancasRivero}, by an independent multinomial colouring with specific parameters. 
In \cite{BlancasRivero}, Blancas and Rivero establish a version of the limit theorems in \cite{Bertoin10}, under a different assumption on the reproduction law: the reproduction law has infinite variance and belongs to the domain of attraction of an $\alpha$-stable distribution, with $\alpha \in (1, 2)$. In particular, they obtain scaling limits for the two-dimensional clone-mutant process, both for the unconditioned process and for the process conditioned on nonextinction, and also establish the convergence, after re-normalization, of the allelic sub-populations towards a tree indexed Jirina branching process.

Combining the multidimensional structure described in \cite{BFH1} for the clone-mutant process associated with the uniform mother-dependent model under the finite-variance reproduction assumption with the infinite alleles stable clone-mutant limit established in \cite{BlancasRivero}, in the present work we prove the scaling limits of the clone-mutant process associated with the general mother-dependent model when the total offspring distribution is critical and belongs to the domain of attraction of an $\alpha$-stable law, while mutations are rare. This regime represents a substantial departure from the finite-variance case: it allows offspring distributions with heavy tails and, consequently, exceptionally large reproductive events that alter both the scaling and the structure of the limits. In particular, we show that the limiting clone masses form a multitype Ji\v{r}ina process driven by \(1/\alpha\)-stable subordinators, whose branching mechanism combines the stable reproduction law with the mother-dependent mutation matrix.

The remainder of the paper is organized as follows. Section \ref{sec:pre} provides details on the general mother-dependent neutral mutations model, including  a description of the clone-mutant Markov chain and its convergence in both the finite variance case and infinite alleles under regular variation. In Section \ref{sec:main}, we establish and prove our main result. We conclude with Section \ref{sec:out}, where we present the most relevant consequences of our main result. 

\section{Preliminaries}\label{sec:pre}

We recall the definition of the general mother-dependent neutral mutations model defined in \cite{BFH2}.  
Let \([d]=\{1,\ldots,d\}\) the set of types, with \(d\geq2\), and let  \(\VEC P=(P_{ij})_{i,j\in[d]}\) be a stochastic matrix satisfying \(P_{ii}=0\) for every \(i\in[d]\). The entry \(P_{ij}\) is the probability that a mutant child
of a type-\(i\) mother has type \(j\). 

Mutations are \emph{neutral}: the type of an individual affects the distribution of the types of its children, but not their total number. Conditional on a
type-\(i\) individual having \(\ell\) children, the types are assigned independently: each child is a clone of its mother, and thus has type \(i\), with probability \(1-r\); or is a mutant of type \(j\neq i\), with probability \(rP_{ij}\), where \(r\in(0,1)\) is the mutation probability.  

Let $\boldsymbol{\xi}^{\,i}=\bigl(\boldsymbol{\xi}^i(1),\ldots,\boldsymbol{\xi}^i(d)\bigr)$
denote the offspring vector of a type-\(i\) individual, where
\(\boldsymbol{\xi}^i(j)\) is the number of its type-\(j\) children.  Consequently, the distribution of $\boldsymbol{\xi}^{\,i}$ is
\begin{equation}\label{eq:offspring}
\boldsymbol{\mu}_i(\VEC v)
=
\mu^+(|\VEC v|)
\binom{|\VEC v|}{\VEC v(1),\ldots,\VEC v(d)}
(1-r)^{v(i)}
\prod_{j\neq i}(rP_{ij})^{\VEC v(j)},
\qquad
\VEC v\in\N_0^d,
\end{equation}
where \(|\VEC v|=\sum_{j=1}^d \VEC v(j)\). 

By construction 
\(\xi^{(+)}\coloneqq|\boldsymbol{\xi}^{\,i}|\sim \mu^+\) for any \(i\in[d]\), and the underlying genealogical structure is that of a  one-type BGW process with offspring distribution \(\mu^+\). The multitype process is obtained by assigning types to individuals according to the mother-dependent mutation rule determined by \(\VEC P\), see \cite[Lemma 3.5]{BFH2}. To describe the population generation by generation, let 
\( (\VEC Y_k:k\in\N_0) \) denote the resulting $d$-type BGW process with branching offspring distribution \((\boldsymbol{\mu}_i, i \in [d])\), where 
\(\VEC Y_k(j)\) is the number of
type-\(j\) individuals in genealogical generation \(k\). Indeed, if
\(\VEC Y_0=\VEC a\in\N_0^d\), then
\[
\VEC Y_{k+1}(j)
=
\sum_{i=1}^d\sum_{\ell=1}^{\VEC Y_k(i)}
\boldsymbol{\xi}_{k,\ell}^{\,i}(j),
\qquad
j\in[d],
\]
where, for each genealogical generation $k$, the offspring vectors  $\boldsymbol{\xi}_{k,\ell}^{\,i}$ are independent and have distribution \(\boldsymbol{\mu}_i\), for each \(i\). 
Hence \((|\VEC Y_k|:k\in\N_0)\) is the underlying one-type BGW process with offspring law \(\mu^+\).

\subsection{The clone--mutant Markov chain}
The genealogical structure is represented by a rooted tree when the population starts from a single individual. Vertices represent individuals and colours represent types. Each non-root vertex is connected to its mother by an edge, and a mark along an edge corresponds to a mutation event; see \cite[Section 3.4]{BFH} for a detailed construction. 
As defined in \cite[Section 3]{BFH2}, we will say that an individual belongs to the \(k\)th \emph{allelic generation} if
exactly \(k\) mutations occur along its ancestral line from the initial population. Thus, an allelic generation is determined by the number of mutation events, not by the number of distinct types encountered along a lineage. 
Individuals connected through ancestral lines containing no mutations form a \emph{clone subfamily}. The mutant children of a clone family are founders of new clone families in the next allelic generation. 

For \(k\in\N_0\) and \(i\in[d]\), let \(\VEC T_k(i)\) be the total
number of type-\(i\) individuals in allelic generation \(k\), and let
\(\VEC M_{k+1}(i)\) be the number of type-\(i\) mutant individuals that
found clone subfamilies in allelic generation \(k+1\). We collect these
quantities in the vectors
\[
\VEC T_k
=
\bigl(\VEC{T}_k(1),\ldots,\VEC{T}_k(d)\bigr),\,
\VEC M_{k+1}
=
\bigl(\VEC{M}_{k+1}(1),\ldots,\VEC{M}_{k+1}(d)\bigr).
\]
As stated in \cite[Proposition 3.1]{BFH2}, the general branching property at the stopping line formed by the mutant
individuals implies that
\(
\left(
(\VEC T_k,\VEC M_{k+1}):k\in\N_0
\right)\)
is a Markov chain starting from \(\VEC M_0=\VEC a\), with transition probabilities satisfying 
\begin{equation}\label{eq:transitionproba}
			\mathbb{P}_{\VEC{a}}(\VEC{T}_k =\VEC{x},\VEC{M}_{k+1}=\VEC{y} \mid \VEC{T}_{k-1}=\VEC{u},\VEC{M}_{k}=\VEC{v} )
			=
			\mathbb{P}_{\VEC{v}} (\VEC{T}_0=\VEC{x},\VEC{M}_1=\VEC{y}),
\end{equation}
		for all $\VEC{u},\VEC{v}, \VEC{x}, \VEC{y}\in\N_0^d$ with
		$\VEC{v} \boldsymbol{\le} \VEC{x}$ in the product order. 
            The clone--mutant chain describes the population through successive mutation events rather than ordinary genealogical generations, and its transitions only depend on the $d$-type BGW process $( \VEC{M}_{k}):k\in\N_0)$. By the definition of the mutant process and the branching property, its easy to see that the process $(|\VEC{M}_{k}|:k\in\N_0)$  is a single-type BGW process, and coincides in distribution with the mutant chain in the infinite allele case. 
	Furthermore, for each $i \in [d]$ and $n\in\N_0$,\begin{equation}\label{eq:condtomcpmatix}
		\VEC{M}_1\big|_{\{\VEC{M}_0=\VEC{e}_i,\,|\VEC{M}_1|=n\}}\sim \mathcal{M}(n,\VEC{P}_{i\cdot}),
	\end{equation}
	where $\mathcal{M}$ denotes the multinomial distribution and $\VEC{P}_{i\cdot}$ the
	$i$-th row of the mutation matrix $\VEC{P}$.

 For the scaling analysis, we will consider a sequence of chains \(\left(
(\VEC T_k^{(n)},\VEC M_{k+1}^{(n)}):k\in\N_0
\right)\) indexed by $n$, with mutation probabilities $r(n)$ and initial populations \(\VEC a^{(n)}\in\N_0^d\). The offspring distribution
\(\mu^+\) and the mutation matrix \(\VEC P\) do not
depend on \(n\). The asymptotic assumptions on
\(r(n)\) and \(\VEC a^{(n)}\) will be specified below.

Next, we recall the scaling limit results established for the mother-dependent model with finite variance and for the infinite-alleles case with regularly varying offspring distributions.

\subsection{The clone-mutant process under for the mother-dependent model with finite variance}

In \cite{BFH1}, we previously studied the asymptotic behaviour of the clone-mutant chain for the (uniform) mother-dependent model, where the associated matrix satisfies $\VEC{P}_{ij} = 1/(d - 1)$ for all $i \neq j$.  The study was conducted  under the following assumptions, which are analogous to those imposed by Bertoin\cite{Bertoin10} for the infinite-alleles case:

\begin{itemize}
    \item[] The total offspring distribution is critical and non-degenerate:
\begin{equation}\label{A1}\tag{A1}
\E[\xi^{(+)}]=1,
\qquad
\mu^+(0)+\mu^+(1)<1; 
\end{equation}
\item[] the offspring variance is finite:
\begin{equation}\label{A2}\tag{A2}  
\sigma^2\coloneqq \text{Var}(\xi^{(+)} )<\infty;
\end{equation}
\item[] for a fixed vector $\VEC{y}\in \R_{+}^d\setminus\{\VEC{0}\}$, the initial population $\VEC{a}^{(n)}\in \N^d_0\setminus\{\VEC{0}\}$ satisfies
\begin{equation}\label{A3}\tag{A3}
		\frac{\VEC{a}^{(n)}(i)}{n} \longrightarrow \VEC{y}(i), \qquad i \in [d]; \text{ and }
\end{equation} 
\item[] the mutation probabilities satisfy
\begin{equation}\label{A4}\tag{A4}
nr(n)\longrightarrow c\in(0,\infty).   
\end{equation}
\end{itemize}
Under assumptions \eqref{A1}--\eqref{A4}, \cite[Theorem 2.8]{BFH1}  establishes that 
\begin{equation} \label{eq:conv-clonemut}
	\left(  \bigl( n^{-2 }\VEC{T}^{(n)}_k,\ n^{-1} \VEC{M}^{(n)}_{k+1} \bigr)  : k\in\N_0 \right) 
	\xRightarrow[n\to\infty]{\mathrm{f.d.d.}}
\bigl( ( \VEC{X}_{k+1}, \VEC{Z}_{k+1}) : k\in\N_0 \bigr),
\end{equation}
where the clone-mutant limit chain is constructed recursively as follows. Set $\VEC{Z}_0\coloneqq \VEC{y}$ such that the mutant-mass vectors can be obtained from nonnegative clone masses; and for $k \in \N_0$,
\begin{enumerate}
	\item[(i)] conditionally on $(\VEC{X}_\ell,\VEC{Z}_\ell)_{\ell \leq k}$, the coordinates
	$\VEC{X}_{k+1}(1),\dots,\VEC{X}_{k+1}(d)$ are independent, with
	\begin{equation*}\label{eq: IG one step}
		\VEC{X}_{k+1}(j) \sim
		\inverseGauss\left(\frac{\VEC{Z}_{k}(j)}{c},\ \frac{\VEC{Z}_{k}(j)^2}{\sigma^2}\right),
		\qquad j \in [d];
	\end{equation*}
	\item[(ii)] 
	\begin{equation*}\label{eq:defZ}
		\VEC{Z}_{k+1}(j)= c[\VEC{X}\VEC{P}]_{j}=\frac{c}{d-1}\sum_{i\neq j}\VEC{X}_{k+1}(i),
		\qquad j \in [d].
	\end{equation*}
\end{enumerate}
In particular, $(\VEC{X}_k : k \in N)$ is a multitype Ji\v{r}ina process with
branching mechanism $\boldsymbol{\kappa}=(\boldsymbol{\kappa}_1,\dots,
\boldsymbol{\kappa}_d)$ given by
\(\boldsymbol{\kappa}_i(\boldsymbol{\lambda})
	= \frac{c}{d-1}\sum_{j\neq i}\kappa(\boldsymbol{\lambda}(j))\) for $\boldsymbol{\lambda}=(\boldsymbol{\lambda}(1),\dots ,\boldsymbol{\lambda}(d))\in \R_+^d $, 
where $\kappa$ is the cumulant of an inverse Gaussian subordinator, that is
\begin{equation*}
    \kappa(q)= \int_0^\infty (1-e^{-qy})\,c^{-1}\nu(\mathrm{d}y)
    = \frac{c}{\sigma^2}\left(\sqrt{ 1+\frac{2q\sigma^2}{c^2}}-1\right),
    \qquad q\in \R_{+}.
\end{equation*}
For further details on multitype Ji\v{r}ina process, see \cite[Section 2.5]{BFH1}.
\subsection{The infinite-alleles clone-mutant process under regular variation}

We now recall the scaling limit obtained by Blancas and Rivero in \cite{BlancasRivero} for the infinite-alleles clone-mutant process when the total offspring distribution has infinite variance. In this setting, the finite-variance assumption  \eqref{A2} in \cite{Bertoin10} is replaced by the following condition. 

Let \(\alpha\in(1,2)\),
\begin{equation}\label{A5}\tag{A5}
\text{the tail distribution of $\xi^{(+)}$ is regularly varying at infinity with index
\(-\alpha\).}    
\end{equation}
According to \cite[Lemma~4.1]{BlancasRivero}, assumption \eqref{A5} yields a normalizing sequence $(L(n): n\geq 0)$, regularly varying at infinity with index $\alpha$, such that,
\[
L(n)\mu^{+}(ndy) \longrightarrow \frac{1}{\Gamma(3-\alpha)}\frac{\mathrm{d}y}{y^{1+\alpha}}
\]
in the sense of vague convergence on $(0,\infty)$. In particular, this implies that
\[
L(n)\,
\E\left[ 1-
 e^{-\lambda\xi^{(+)}/n}- \frac{\lambda\xi^{(+)}}{n}
\right]
\longrightarrow\lambda^\alpha,
\qquad \lambda\geq0.
\]
In this case, the initial population must also be rescaled according to this normalizing sequence. Instead of assumption~\eqref{A3}, the authors impose the following condition.
\begin{itemize}
\item[]
For some $y\in(0,\infty)$, the initial population $a^{(n)}$
satisfies
\begin{equation}\label{A6}\tag{A6}
    \frac{a^{(n)}}{L(n)r(n)}\longrightarrow y.
\end{equation}
\end{itemize}
Notice that assumption~\eqref{A5} together with assumption~\eqref{A4} imply that
$L(n)r(n)$ is regularly varying with index \(\alpha-1>0\). In particular,
\begin{equation}\label{eq:mutRVgoinf}
L(n)r(n)\longrightarrow\infty \qquad \text{and} \qquad  \frac{L(n)r(n)}{n}\longrightarrow 0   
\end{equation}
Thus assumption~\eqref{A6} describes an initial population whose number of individuals is of the same order as the number of mutants produced in the stable rare-mutation regime. 

\begin{remark}\label{re:scale}The constant $c$ in \eqref{A4} relates \(r(n)\) to the scale \(n^{-1}\). However, in this case, mutant counts are normalized by
\(L(n)r(n)\), rather than by \(L(n)/n\). Consequently, \(c\) is already absorbed
into the normalization and does not appear in the limiting type-allocation map. If mutant counts were instead normalized by \(L(n)/n\), the corresponding limiting mutant vector would contain the additional factor \(c\).
\end{remark}
Under assumptions~\eqref{A1}, \eqref{A4}--\eqref{A6}, \cite[Proposition~4.5]{BlancasRivero} establishes convergence of
finite-dimensional distributions of the appropriately normalized
clone-mutant process. In the normalization and notation of that result,
this reads
\begin{equation}\label{eq:stable-clonemut}
\left(
\left(
\frac{T_k^{(n)}}{L(n)},
\frac{M_{k+1}^{(n)}}{L(n)r(n)}
\right):k\in\N_0
\right)
\xRightarrow[n\to\infty]{\mathrm{f.d.d.}}
\left(
(Z_{k+1}^{1/\alpha},Z_{k+1}^{1/\alpha}):k\in\N_0
\right),
\end{equation}
where $(Z_k^{1/\alpha}:k\in\N)$ is a Ji\v{r}ina process
with reproduction measure
\begin{equation}\label{eq:stablemeasure}
\nu^\alpha(\mathrm{d}z)
=
\frac{1}{\alpha\Gamma(1-\alpha^{-1})}
z^{-1-1/\alpha}\,\mathrm{d}z,
\qquad z>0.
\end{equation}
Equivalently, its branching mechanism is the Laplace exponent of a
$1/\alpha$-stable subordinator:
\[
\int_0^\infty (1-e^{-qz})\,\nu^\alpha(\mathrm{d}z)
=q^{1/\alpha},
\qquad q\geq0.
\]

The convergence in \eqref{eq:stable-clonemut} shows that, under the
corresponding normalizations, the total size of an allelic family and
the number of mutants it produces have the same limiting value.
In contrast with the finite-variance case, the limiting reproduction
mechanism is stable rather than inverse Gaussian.

The convergence results \eqref{eq:conv-clonemut} and \eqref{eq:stable-clonemut} motivate the extension developed in the next section.
We consider a finite number of allelic types, with mutation
probabilities depending on the mother's type, and investigate the
clone-mutant process under the regular-variation assumption~\eqref{A5}.
The main question is how the finite-type mutation structure enters
the limiting multitype Ji\v{r}ina process, while the heavy-tailed
reproduction law determines its stable branching mechanism.

\section{Scaling limits for the finite-alleles clone--mutant process}\label{sec:main}

In the case considered here we retain the criticality and regular-variation assumptions \eqref{A1} and \eqref{A5}. Let $(L(n): n\geq 0)$ be the normalizing sequence
specified above. For the $d$-dimensional clone-mutant chain, we replace assumption~\eqref{A6} by the following initial-population condition. For some $\VEC y\in\R_+^d\setminus\{\VEC0\}$,
\begin{equation}\label{A7}\tag{A7}
    \frac{\VEC a^{(n)}(i)}{L(n)r(n)}\longrightarrow y(i),
  \qquad i\in[d]. 
\end{equation}
By \ref{re:scale}, this is equivalent to $\VEC a^{(n)}(i)/L(n)n\to cy(i)$ for every $i\in[d]$.
Thus, the normalization of the initial population accounts for the factor $c$ in the limiting branching mechanism.

In analogy with \cite[Definition 2.6]{BFH1}, we now introduce the \textit{stable clone--mutant limiting chain}. The construction follows
the same recursive principle as in the finite-variance case. At each allelic generation, the mutant mass provides the founders of the next generation of clone families. These families evolve
independently, and the mutant mass they produce is distributed among the $d$ types according to the mutation matrix $\VEC P$.
The main difference lies in the distribution of the clone masses, which is now governed by stable subordinators rather than inverse-Gaussian subordinators. In particular, as in the finite-variance case, the transition to the next allelic generation depends on the present state only through its mutant coordinate $\VEC v$. 

\begin{definition}
\label{def:stable-limit-chain}
Let \(\VEC y\in\R_+^d\setminus\{\VEC0\}\).
The clone--mutant limiting chain in the stable regime started from \(\VEC y\) is the \(\R_+^{2d}\)-valued process
\[
\bigl((\VEC X_k,\VEC Z_k):k\in\N\bigr),
\]
constructed recursively as follows.
Set \(\VEC Z_0=\VEC y\) and, for \(k\in\N_0\),
\begin{enumerate}
\item[(i)]
Conditionally on the preceding allelic generations,
the coordinates
\(\VEC X_{k+1}(1),\ldots,\VEC X_{k+1}(d)\)
are independent and
\begin{equation}\label{eq:stable-one-step}
\E\left[
e^{-\lambda\VEC X_{k+1}(j)}
\mid\VEC Z_k
\right]
=
\exp\left\{
-\VEC Z_k(j)c\kappa_c(\lambda)
\right\},
\quad \lambda\geq0,\ j\in[d],
\end{equation}
where $\kappa_c$ is the Laplace exponent of the first passage time process of a spectrally positive $\alpha$-stable subordinator with drift $c$, that is, \(\kappa_c(\lambda)\) is the unique nonnegative solution of
 \begin{equation}\label{eq:substadift}
     \kappa_c(\lambda)^\alpha
+c\kappa_c(\lambda)=\lambda.
 \end{equation}
\item[(ii)]
\begin{equation}\label{eq:stable-defZ}
\VEC Z_{k+1}
= \VEC X_{k+1}\VEC P,
\qquad \Longleftrightarrow \qquad \VEC Z_{k+1}(j)
=
\sum_{i\neq j}\VEC X_{k+1}(i)\VEC{P}_{ij},
\qquad j\in[d].
\end{equation}
\end{enumerate}
\end{definition}

\begin{remark}
Setting $c=0$ in \ref{eq:substadift} gives the Laplace exponent of the first-passage-time of an $1/\alpha$-stable subordinator associated with a spectrally positive $\alpha$-stable Lévy process with measure given in \eqref{eq:stablemeasure},  In terms of the model, the constant $c$ enters through the limiting relation between clone masses
and mutant masses, $\mathbf Z=c\,\mathbf X\mathbf P$. In the
first-passage representation, this contribution appears as the linear
drift term $c\kappa_c(\lambda)$. 
\end{remark}

We can now state our main scaling-limit theorem.
\begin{theorem}\label{thm:chain}
Suppose that Assumptions~\eqref{A1},\eqref{A5} and \eqref{A7} hold. Then
\begin{equation}\label{eq:clone-mutant-convergence}
\left(
\left(
\frac{\VEC T_k^{(n)}}{L(n)},
\frac{\VEC M_{k+1}^{(n)}}{L(n)r(n)}
\right):k\in\N_0
\right)
\xRightarrow[n\to\infty]{\mathrm{f.d.d.}}
\left(
(\VEC X_{k+1},\VEC Z_{k+1}):
k\in\N_0
\right).
\end{equation}
Moreover, \((\VEC Z_k:k\in\N_0)\) is a
multitype Ji\v{r}ina process with branching mechanism
\begin{equation}\label{eq:GMFk}
K_i(\boldsymbol{\lambda})
=
c \kappa_c\left(
\sum_{j=1}^dP_{ij}\boldsymbol{\lambda}(j)
\right),
\qquad i\in[d].
\end{equation}
\end{theorem}

\begin{proof}    
We prove the result by induction on the allelic generation. First, we establish convergence for the initial allelic generation, corresponding to the pair $(\VEC{T}_0^{(n)}, \VEC{M}_1^{(n)})$, starting from a single individual of type $i\in[d]$. All individuals  are of type $i$ up to the first mutation along a lineage. Therefore, after forgetting the
types assigned at mutation, the pair formed by its total clone-family size and its total number of mutant children has exactly the law of the pair $(T_0^{(n)}, M_1^{(n)})$ considered in \cite[Lemma 4.7(i)]{BlancasRivero}.

Let $\VEC s,\VEC t\in\R_+^d$. Conditionally on the total number $\vert \VEC M_1^{(n)} \vert$ of mutants produced by a type-$i$ clone family, their types are independent with common distribution
$(P_{i1},\ldots,P_{id})$. Consequently,
\begin{equation}\label{eq:conditional-colouring-transform}
\begin{split}
    &\E_{\VEC e_i}^{r(n)}
 \left[
 \exp\left\{
 -\frac{\VEC s(i)}{L(n)} \VEC T_0^{(n)}(i)
 -\frac{\langle\VEC t,\VEC M_1^{(n)} \rangle}
        {L(n)r(n)}
 \right\}
 \,\middle|\,\VEC T_0^{(n)}(i), \vert \VEC M_1^{(n)} \vert
 \right]\\
&\quad =
 \exp\left\{-\frac{\VEC s(i)}{L(n)}\VEC T_0^{(n)}(i) \right\}
 \left(
 \sum_{j=1}^d
 P_{ij}\exp\left\{-\frac{\VEC t(j)}{L(n)r(n)}\right\}
 \right)^{ \vert \VEC M_1^{(n)} \vert } .
\end{split}
\end{equation}
Hence, we can write
\begin{equation*}\label{eq:condcoltrans}
\begin{split}
&\E_{\VEC e_i}^{r(n)}
 \left[
 \exp\left\{
 -\frac{\VEC s(i)}{L(n)} \VEC T_0^{(n)}(i)
 -\frac{\langle\VEC t,\VEC M_1^{(n)} \rangle}
        {L(n)r(n)}
 \right\}
 \right]
 \notag\\
&\quad =
\E_{\VEC e_i}^{r(n)}
 \left[ e^{-\frac{\VEC s(i)}{L(n)}\VEC T_0^{(n)}(i)}
\exp\left\{ - \frac{ \vert \VEC M_1^{(n)} \vert }{ L(n)r(n)} \log
\left(
 \sum_{j=1}^d
 P_{ij}\exp\left\{-\frac{\VEC t(j)}{L(n)r(n)}\right\}
 \right)^{-L(n)r(n)} \right\}  \right].    
\end{split}
\end{equation*}
By using \eqref{eq:mutRVgoinf}, we have that
\begin{equation}\label{eq:convP}
-L(n)r(n)
\log\left(
\sum_{j=1}^d
P_{ij}\exp\left\{-\frac{\VEC t(j)}{L(n)r(n)}\right\}
\right)
\longrightarrow
\sum_{j=1}^dP_{ij}\VEC t(j).    
\end{equation}
Therefore, if we apply \eqref{eq:convP} together with the infinite--alleles result \cite[Lemma~4.7(i)]{BlancasRivero}  to \eqref{eq:conditional-colouring-transform}, we obtain
\begin{equation}
\E_{\VEC e_i}^{r(n)}
\left[
\exp\left\{
-\frac{\VEC s(i)}{L(n)} \VEC T_0^{(n)}(i)
-\frac{\langle\VEC t,\VEC M_1^{(n)}\rangle}
       {L(n)r(n)}
\right\}
\right]
\longrightarrow
\left(
\VEC s(i)+\sum_{j=1}^dP_{ij}\VEC t(j)
\right)^{1/\alpha}.
\label{eq:one-founder-joint-exponent}
\end{equation}
For an arbitrary initial population
\(\VEC a^{(n)}\), we use the general branching property to obtain that
\begin{equation*}\label{eq:first-generation-product}
\begin{split}
&\E_{\VEC a^{(n)}}^{r(n)}
\left[
\exp\left\{
-\frac{\langle\VEC s,\VEC T_0^{(n)}\rangle}{L(n)}
-\frac{\langle\VEC t,\VEC M_1^{(n)}\rangle}
       {L(n)r(n)}
\right\}
\right]\\
&\quad =
\prod_{i=1}^d
\left[
\E_{\VEC e_i}^{r(n)}
\left[
\exp\left\{
-\frac{\VEC s(i)}{L(n)}\VEC T_0^{(n,i)}
-\frac{\langle\VEC t,\VEC M_1^{(n,i)}\rangle}
       {L(n)r(n)}
\right\}
\right]
\right]^{\VEC a^{(n)}(i)},
\end{split}
\end{equation*}
and applying \eqref{eq:one-founder-joint-exponent} together with \eqref{A7} we obtain
\begin{equation*}\label{eq:first-generation-limit-transform}
\E_{\VEC a^{(n)}}^{r(n)}
\left[
\exp\left\{
-\frac{\langle\VEC s,\VEC T_0^{(n)}\rangle}{L(n)}
-\frac{\langle\VEC t,\VEC M_1^{(n)}\rangle}
       {L(n)r(n)}
\right\}
\right]\longrightarrow
\exp\left\{
-\sum_{i=1}^d \VEC y(i)
\left(
\VEC s(i)+\sum_{j=1}^dP_{ij}\VEC t(j)
\right)^{1/\alpha}
\right\}.
\end{equation*}
The induction step follows directly from the transition probabilities in \eqref{eq:transitionproba}, as in \cite[Lemma~2.2]{BFH1}.
\end{proof}

\section{Further implications on stable limits of the allelic genealogy}\label{sec:out}

The convergence of the clone--mutant Markov chain established here provides a starting point for studying the asymptotic genealogical structure of the allelic subfamilies. 

A natural extension is to establish convergence of the multitype tree of alleles, following the approaches developed in \cite{BlancasRivero} for branching processes with infinite-variance offspring distributions and in \cite{BFH1} for the finite-variance mother-dependent model. More precisely, one can work with the $d$-type tree of alleles constructed in \cite[Section 4.2]{BFH1}.  The scaling limit obtained here suggests that the rescaled tree of alleles should converge, in the sense of finite-dimensional distributions, to a tree-indexed multitype Ji\v{r}ina process, see \cite[Definition 2.10]{BFH1}. As in the finite-variance case, the limiting subfamilies can be described through the ranked atoms of Poisson random measures. The key difference is that the inverse-Gaussian L\'evy measure is replaced by the L\'evy measure of the stable subordinator arising in the present regime. The mutation matrix $\VEC P$ determines the types
of the descendant subfamilies.

A further direction is to investigate the scaling limit of the clone--mutant chain conditioned on non-extinction.
In the infinite-alleles setting, \cite{BlancasRivero} established that, under the $Q$-process, the rescaled
clone--mutant chain converges to a discrete-time
continuous-state branching process with immigration.
In the stable regime, the limiting reproduction
mechanism is determined by the L\'evy measure
$\nu_\alpha$, whereas the immigration mechanism is governed by $z\nu_\alpha(dz)$.
The corresponding conditioned genealogy can be interpreted in terms of a tree of alleles with immigration.
Theorem \ref{thm:chain} suggests a multitype analogue for the general mother-dependent model. In particular, the
$Q$-process of the multitype clone--mutant chain,
whose construction and genealogical representation
are studied in \cite{BFH2}, provides a natural
starting point for investigating the stable scaling
limit under conditioning on non-extinction.
Under our hypotheses, one expects a multitype Ji\v{r}ina process with immigration, in which the mutation matrix $\VEC P$ determines the distribution of offspring masses across types.

\bibliography{biblio.bib}
\bibliographystyle{plain}

\end{document}